\documentclass[11pt]{amsart}  
\usepackage{amssymb, eucal}
\usepackage[all,arc]{xy}
\usepackage{enumerate}
\usepackage{color}

\newenvironment{tfae}{
\begin{enumerate}}{\end{enumerate}}

\newtheorem{prop}{Proposition}[section]
\newtheorem{lemma}[prop]{Lemma}
\newtheorem{theorem}[prop]{Theorem}
\newtheorem{corollary}[prop]{Corollary}
\theoremstyle{definition}
\newtheorem{definition}[prop]{Definition}

\newtheorem{examples}[prop]{Examples}
\newtheorem{remark}[prop]{Remark}

\counterwithin*{equation}{section}

\newdir{ >}{{}*!/-8pt/@{>}}

\def\mathrmdef#1{\expandafter\def\csname#1\endcsname{{\rm#1}}}
\mathrmdef{can}\mathrmdef{ev}\mathrmdef{Ker}\mathrmdef{id}\mathrmdef{Id}\mathrmdef{lex}\mathrmdef{op}\mathrmdef{pr}
\mathrmdef{prod}\mathrmdef{Ult}\mathrmdef{Aut}\mathrmdef{End}

\def\mathsfdef#1{\expandafter\def\csname#1\endcsname{{\sf#1}}}
\mathsfdef{F}
\mathsfdef{P}
\mathsfdef{L}
\mathsfdef{Gp}
\mathsfdef{oGp}
\mathsfdef{roGp}
\mathsfdef{R}
\mathsfdef{U}

 \def\mathbfdef#1{\expandafter\def\csname#1\endcsname{{\rm\bf#1}}}

\mathbfdef{Alg} \mathbfdef{Bool} \mathbfdef{C} \mathbfdef{CLC}
\mathbfdef{Comp}\mathbfdef{Conn}\mathbfdef{CompHaus}\mathbfdef{ContLat}
\mathbfdef{Cont}\mathbfdef{D}\mathbfdef{Fam} \mathbfdef{Ps} \mathbfdef{Grp}
\mathbfdef{Haus}\mathbfdef{Inf}
\mathbfdef{Lat}\mathbfdef{LC}\mathbfdef{Mon} \mathbfdef{MultiOrd}
\mathbfdef{Ord}\mathbfdef{RelAlg}\mathbfdef{PsTop}\mathbfdef{Rel}
\mathbfdef{Set} \mathbfdef{SGrp} \mathbfdef{Sup} \mathbfdef{Top}
\mathbfdef{OrdGrp}\mathbfdef{ROrdGrp} \mathbfdef{OrdAb}
\mathbfdef{CMon}

\def\bkN{\mathbb{N}}
\def\bkZ{\mathbb{Z}}

\begin{document}  
\title{Right-preordered groups from a categorical perspective}  
\author{Maria Manuel Clementino}
\address{CMUC, Department of Mathematics, University of Coimbra, 3000-143 Coimbra, Portugal}\thanks{}
\email{mmc@mat.uc.pt}

\author{Andrea Montoli}
\address{Dipartimento di Matematica ``Federigo Enriques'', Universit\`{a} degli Studi di Milano, Via Saldini 50, 20133 Milano, Italy}

\email{andrea.montoli@unimi.it}
\thanks{}

\begin{abstract}
We study the categorical properties of right-preordered groups, giving an explicit description of limits and colimits in this category, and studying some exactness properties. We show that, from an algebraic point of view, the category of right-preordered groups shares several properties with the one of monoids. Moreover, we describe split extensions of right-preordered groups, showing in particular that semidirect products of ordered groups have always a natural right-preorder.
\end{abstract}
\maketitle  
\section{Introduction}
In \cite{CMFM} preordered groups have been studied from a categorical point of view. A preordered group is a group equipped with a preorder (i.e. a reflexive and transitive relation) which is compatible with the group operation both on the left and on the right, in the sense that the group operation is monotone with respect to such preorder. In particular, limits and colimits in the category \OrdGrp\ of preordered groups and monotone group homomorphisms have been explicitly described, and some exactness properties of \OrdGrp\ have been explored. From an algebraic point of view, \OrdGrp\ does not share with the category \Grp\ of groups the strong properties that allow the intrinsic interpretation of homological algebra, like being a protomodular \cite{Bourn protomod} or a Mal'tsev \cite{CLP Maltsev} category. From this perspective, \OrdGrp\ turns out to be more similar to the category \Mon\ of monoids, where such properties hold only ``locally'', i.e. for some ``good'' objects. In the case of $\OrdGrp$, such good objects are those preordered groups whose preorder is symmetric (and hence an equivalence relation). Another interesting aspect of this relatively weak algebraic context is that, on a split extension of groups, whose kernel and codomain are preordered groups, there may be many compatible preorders, turning it to a split extension in $\OrdGrp$, or none. \\

The wider class of the so-called right- (or left-)preordered groups is also interesting for applications in various mathematical fields. These are groups equipped with a preorder which is compatible with the group operation only on the right (or on the left): if $a \leq b$ then, for all $c$, $a+c \leq b+c$ (we are using the additive notation although our groups need not be abelian). In fact, right-orders on a group are related to actions of the group on the real line (see, e.g., \cite{Ghys, ClayRolfsen}). Moreover, right-ordered groups are used to give a description of the free lattice-ordered groups on given groups \cite{Conrad}. Following a similar spirit, in \cite{ColacitoMarra} spaces of right-orders on partially ordered groups have been related to spectral spaces of lattice-ordered groups. In \cite{Glass} several important examples of preorders on groups that are compatible with the group operation only on one side are described. \\

The aim of this paper is to extend the study made in \cite{CMFM} to the category \ROrdGrp\, whose objects are the right-preordered groups and whose arrows are the monotone group homomorphisms. In particular, we observe that \ROrdGrp\ is isomorphic to the category whose objects are pairs $(G,M)$, where $G$ is a group and $M$ is a submonoid of $G$, and whose arrows are group homomorphisms that (co)restrict to the submonoids. Using the good properties of the forgetful functors from \ROrdGrp\ to, on the one hand, \Grp\, and, on the other hand, the category \Ord\ of preordered sets and monotone maps, as well as the ones of the functor associating to every right-preordered group its positive cone, we give a description of limits and colimits in $\ROrdGrp$, and we show that this category has the same exactness properties as $\OrdGrp$. Moreover, we characterize the ``good'' objects, from an algebraic point of view, proving that they are still the groups equipped with a (right-compatible) equivalence relation. Finally, we explore the possible compatible right-preorders on split extensions, showing that, as for \OrdGrp, all such preorders are bounded by the product preorder (a pair is positive if and only if both components are) and the so-called lexicographic preorder. We prove that the existence of a compatible right preorder is equivalent to the fact that the lexicographic one is compatible (extending a result of \cite{CR}). Using the semidirect product construction, we exhibit examples of split extensions which admit compatible right-preorders without admitting preorders that are compatible on both sides.

\section{The category \ROrdGrp\ of right-preordered groups} \label{section first properties}

A \emph{right-preordered group} is a group $X$ together with a preorder (i.e. a reflexive and transitive relation) $\leq$ such that
\[ \forall \ x,y,z\in X, \qquad x\leq y \quad \Rightarrow \quad x+z \leq y+z. \]
A morphism of right-preordered groups is a monotone group homomorphism. We denote the category of right-preordered groups and their morphisms by $\ROrdGrp$. We point out that, when the group is abelian, the seemingly weaker condition of being right-preordered coincides with being a preordered group.

A right-preorder on a group $X$ determines a submonoid of $X$, namely $P_X=\{x\in X\,|\,x\geq 0\}$, also known as the \emph{positive cone} of $X$.

\begin{prop}
For a group $X$, right-preorders on $X$ are in bijective correspondence with submonoids of $X$.
\end{prop}

\begin{proof}
If $\leq$ is a right-preorder on $X$, then, if $x\geq 0$ and $y\geq 0$, $x+y\geq 0+y\geq 0$, hence $P_X=\{x\in X\,|\,x\geq 0\}$ is a submonoid of $X$.

Conversely, given a submonoid $M$ of $X$, define $\leq$ on $X$ by $x\leq y$ if $y-x\in M$. Then $\leq$ is clearly reflexive and transitive: $x\leq y$ and $y\leq z$ implies $y-x\in M$ and $z-y\in M$, hence $(z-y)+(y-x)=z-x\in M$, that is, $x\leq z$. Moreover, for every $x,y,z\in X$, if $y-x\in M$ then $y+z-z-x=(y+z)-(x+z)\in M$; that is, $x\le y$ implies $x+z\leq y+z$ as claimed.
\end{proof}

\begin{remark}
Given two right-preordered groups $X$ and $Y$, a group homomorphism $f\colon X\to Y$ is monotone exactly when $f(P_X)\subseteq P_Y$. Hence the category $\ROrdGrp$ is isomorphic to the category having as objects pairs $(X,M)$, where $X$ is a group and $M$ is a submonoid of $X$, and as morphisms $f\colon(X,M)\to(Y,N)$ group homomorphisms $f\colon X\to Y$ that (co)restrict to $M\to N$. For simplicity, herein we will refer to both the right-preorder on $X$ and its positive cone as a right-preorder on $X$.
\end{remark}

In order to study the behaviour of the category $\ROrdGrp$, we start by observing the following.

\begin{prop}
The (full) inclusion of the category $\OrdGrp$ of preordered groups into\linebreak $\ROrdGrp$ has a left adjoint.
\end{prop}
\begin{proof}
Given a right-preordered group $X$, with positive cone $P_X$, form the least submonoid $\widetilde{P}_X$ of $X$ closed under conjugation and containing $P_X$. Define $F\colon\ROrdGrp\to\OrdGrp$ by $F(X,P_X)=(X,\widetilde{P}_X)$ and $F(f)=f$; it is easy to check that from $f(P_X)\subseteq P_Y$ it follows that $f(\widetilde{P}_X)\subseteq\widetilde{P}_Y$. The identity maps $(X,P_X)\to(X,\widetilde{P}_X)$ are, by construction of $\widetilde{P}_X$, morphisms in $\ROrdGrp$ which define pointwise the unit of the adjunction.
\end{proof}

The functor $\P \colon \ROrdGrp \to \Mon$, which sends a right-preordered group to its positive cone, factors through the category $\Mon_\can$
of monoids with cancellation, since every submonoid of a group satisfies both left- and right-cancellation properties. The functor
$\P_0 \colon \ROrdGrp \to \Mon_\can$ has a left adjoint:
\[ \xymatrix{ \Mon_\can \ar@/^/[rr]^{\roGp} \ar@{}[rr]|-{\bot}& & \ROrdGrp, \ar@/^/[ll]^{\P_0} } \]
which is constructed by considering the group completion $\Gp(M)$ of any cancellative monoid $M$, with the preorder determined by the image of $M$ into $\Gp(M)$ via the unit of the group completion adjunction. The verification that this construction gives the left adjoint to $\P_0$ is essentially the same as the one described in \cite{CMFM} for preordered groups. Composing this adjunction with the one described in the previous proposition, we get precisely the adjunction (A) in \cite{CMFM}. Following the same arguments as in \cite[Proposition 3.1]{CMFM}, we conclude that any monoid $M$ which is embeddable in a group is embeddable in its group completion, and this is enough to conclude that it is the positive cone of $\roGp(M)$. \\

It is well-known that the forgetful functor $\Grp\to\Set$ is monadic and the forgetful functor $\Ord\to\Set$ is topological. As for the case of $\OrdGrp$ (see \cite[Proposition 2.3]{CMFM}) we have the following results, whose proofs follow arguments similar to those used for $\OrdGrp$.

\begin{prop}\label{top}
Consider the forgetful functors $\U_1\colon\ROrdGrp\to\Grp$ and $\U_2\colon\ROrdGrp\to\Ord$,
with $\U_1$ forgetting the preorder and $\U_2$ the group structure.
The functor $\U_1$ is topological while $\U_2$ is monadic. We have therefore the following commutative diagram
\[\xymatrix{&\ROrdGrp\ar[ld]_{\mbox{(topological) }\U_1\;}\ar[rd]^{\;\U_2 \mbox{ (monadic)}}\\
\Grp\ar[dr]_{\mbox{(monadic) } |\;\;|\;}&&\Ord\ar[dl]^{\;|\;\;|\mbox{ (topological)}}\\
&\Set}\]
\end{prop}
\begin{proof}
\emph{$\U_1$ is a topological functor}: let $(f_i\colon X\to X_i)_{i\in I}$ be a family of group homomorphisms where each $X_i$, for $i\in I$, is a right-preordered group. Then $P_X=\{x\in X\,|\, f_i(x)\in P_{X_i}$ for every $i\in I\}$ is a submonoid of $X$ such that $f_i(P_X)\subseteq P_{X_i}$, for all $i\in I$. This defines the $\U_1$-initial lifting for $(f_i)$.

\emph{$\U_2$ is a monadic functor}: To prove this we will use \cite[Theorem 2.4]{VAtlantis}.

\noindent (a) \emph{$\U_2$ has a left adjoint} $\L_2 \colon \Ord \to \ROrdGrp$: $\L_2$ assigns to each preorder $A$ the free group $\F(A)$ on the set $A$ equipped with the right-preorder induced by the submonoid of $\F(A)$ generated by the elements of the form $b-a$ for all $a,b\in A$ with $a\leq b$. It is easy to check that $\L_2$ is a functor which is left adjoint to $\U_2$.

\noindent (b) \emph{$\U_2$ reflects isomorphisms}: a morphism $f\colon X\to Y$ in \ROrdGrp\ with $\U_2(f)$ an isomorphism in \Ord\ is a bijective homomorphism whose inverse is both a homomorphism and monotone, hence $f$ is an isomorphism in \ROrdGrp.

\noindent (c) \emph{\ROrdGrp\ has and $\U_2$ preserves coequalizers of all $\U_2$-contractible coequalizer pairs.} First of all \ROrdGrp\ is cocomplete, since it is topological over a cocomplete category. Given morphisms $f,g\colon X\to Y$ in \ROrdGrp\ with $\U_2(f),\U_2(g)$ a contractible pair in \Ord, their coequalizer $q\colon Y\to Q$ in \ROrdGrp\ is preserved by $\U_1$, and so also by $|\;\;|\cdot \U_1$ because $|\;\;|$ is monadic and $|\U_1(f)|, |\U_1(g)|$ form a contractible pair in \Set. Therefore $\U_2(q)$ is the coequalizer of $|\U_2(f)|, |\U_2(g)|$ in \Ord, since it is a split epimorphism and $|\U_2(q)|$ is the coequalizer of $\U_2(f),\U_2(g)$ in \Set.
\end{proof}

The properties of $\U_1$ and $\U_2$ give important information on the categorical behaviour of \ROrdGrp.

\begin{remark}\label{re:properties}
\begin{enumerate}
\item Topologicity of $\U_1\colon\ROrdGrp\to\Grp$ guarantees that it has both a left and a right adjoint. The former one equips a group $X$ with the discrete order, so that $P_X=\{0\}$, while the latter one equips a group with the total preorder (so that $P_X=X$). Moreover, with $\Grp$ complete and cocomplete, also $\ROrdGrp$ is complete and cocomplete.

\item Both $\U_1\colon\ROrdGrp\to\Grp$ and $\U_2\colon\ROrdGrp\to\Ord$ preserve limits, which gives us a complete description of limits in $\ROrdGrp$: algebraically they are formed like in $\Grp$, and then equipped with the limit preorder.

\item The forgetful functor $\P\colon \ROrdGrp\to\Mon$ which assigns to each right-preordered group its positive cone, and to each morphism its (co)restriction to the positive cones, preserves limits and coproducts, but not coequalizers.
Indeed, as for \OrdGrp, the positive cone of a product of right-preordered groups is the product of their positive cones, and analogously for equalizers. For coproducts the situation differs from what happens in \OrdGrp, where the coproduct of positive cones does not need to be closed under conjugation as a submonoid of the coproduct of the preordered groups (see, for instance, \cite[Example 2.10]{CMFM}). On the contrary, in \ROrdGrp\ the positive cone of a coproduct is the coproduct of the positive cones, just because the coproduct of submonoids is a submonoid of the coproduct.

We point out that the functor $\P$ does not preserve coequalizers; for instance, the coequalizer of the pair of morphisms $f,g\colon(\bkZ,0)\to(\bkZ,\bkN)$, with $f(1)=1$ and $g(1)=2$, is the constant morphism into $\{0\}$, while the coequalizer of $\P(f),\P(g)\colon\{0\}\to\bkN$ is the identity on $\bkN$.

\item Analogously to the case of \OrdGrp\ (see \cite[Remark 2.4]{CMFM}), given a morphism $f\colon (X,P_X)\to (Y,P_Y)$ in \ROrdGrp:
\begin{enumerate}
\item $f$ is an epimorphism if and only if $f$ is surjective (and epimorphisms are stable under pullback);
\item $f$ is a regular epimorphism if and only if both $f$ and $\P(f)$ are surjective (and regular epimorphisms coincide with pullback-stable regular epimorphisms, and with strong epimorphisms, and with extremal epimorphisms);
\item $f$ is a monomorphism if and only if $f$ is injective;
\item $f$ is a regular monomorphism if and only if $f$ is injective and $P_X=f^{-1}(P_Y)$.
\end{enumerate}
This gives us two proper and stable factorization systems in $\ROrdGrp$, \emph{(Epi, Reg Mono)} and \emph{(Reg Epi, Mono)}, factoring each morphism as outlined in the following diagram:
\[\xymatrix{(X,P_X)\ar[rr]^-f\ar[dr]_e&&(Y,P_Y)&(X,P_X)\ar[rr]^-f\ar[rd]_{e'}&&(Y,P_Y)\\
&(f(X),P_Y\cap f(X))\ar[ru]_m&&&(f(X),f(P_X))\ar[ru]_{m'}}\]
\end{enumerate}
\end{remark}

Following the arguments used in the proofs of \cite[Proposition 2.5 and Remark 2.6]{CMFM}, we can easily deduce that:
\begin{prop}
The category \ROrdGrp:
\begin{enumerate}
\item is a regular category;
\item is a normal category;
\item is an efficiently regular -- but not exact -- category;
\end{enumerate}
\end{prop}

We refer to \cite{Barr, ZJanelidze2010, Bourn Baersums} for the definitions of exact, regular, normal, and efficiently regular category.

As outlined in \cite[Propositions 2.7, 2.8]{CMFM}, in every efficiently regular category the change-of-base functor induced by a regular epimorphism is monadic. Therefore:

\begin{corollary}
In \ROrdGrp\ a morphism is effective for descent if and only if it is a regular epimorphism.
\end{corollary}

\begin{remark}
As observed in Remark \ref{re:properties}, exactly as for \OrdGrp, a morphism $f\colon X\to Y$ in \ROrdGrp\ is a strong epimorphism (or extremal, or regular, epimorphism) if and only if both $f$ and its (co)restriction to the positive cones are surjective.

In \cite[Remark 2.9]{CMFM} we pointed out that this does not extend to families of morphisms in \OrdGrp; that is, for a family of morphisms, to be jointly strongly epimorphic in \OrdGrp\ does not imply the family of (co)restrictions to the positive cones to be jointly strongly epimorphic. The situation in \ROrdGrp\ differs from this one: a family of morphisms $(f_i\colon X_i\to X)_{i\in I}$ is jointly strongly epimorphic in \ROrdGrp\ if and only if
\begin{enumerate}
\item $(f_i\colon X_i\to X)_{i\in I}$ is jointly strongly epimorphic in \Grp;
\item $(f_i\colon P_{X_i}\to P_X)_{i\in I}$ is jointly strongly epimorphic in \Mon.
\end{enumerate}
Indeed, if (1) does not hold, given a factorization of $(f_i)$ through a monomorphism in \Grp\
\[\xymatrix{X_i\ar[rr]^{f_i}\ar[rd]_{g_i}&&X\\
&Y\ar[ru]_m}\]
we can equip $Y$ with $P_Y=m^{-1}(P_X)$, making this factorization live in \ROrdGrp. If (2) does not hold, so that we can obtain a factorization
\[\xymatrix{P_{X_i}\ar[rr]^{P(f_i)}\ar[rd]_{P(g_i)}&&P_X\\
&M\ar[ru]_m}\]
with $m$ monic (and not an isomorphism), we can factor $(f_i)$ in \ROrdGrp\ as follows
\[\xymatrix{(X_i,P_{X_i})\ar[rr]^{f_i}\ar[rd]_{f_i}&&(X,P_X)\\
&(X,M)\ar[ru]_{1_X}}\]
showing that $(f_i)$ is not jointly strongly epimorphic. The converse implication is obvious.
\end{remark}

\section{Algebraic properties of $\ROrdGrp$} \label{section algebraic properties}
It was observed in \cite{CMFM} that the category $\OrdGrp$ of preordered groups is unital \cite{Bourn unital} and weakly protomodular, but not protomodular \cite{Bourn protomod}. The same arguments work for $\ROrdGrp$, showing that the category of right-preordered groups has the same properties. As it was done in \cite{CMFM} for $\OrdGrp$, it becomes then interesting to characterize those objects that, locally, have a stronger categorical-algebraic behaviour. In order to do that, we recall the following notions.

\begin{definition}
A point (i.e. a split epimorphism with a fixed section) $\xymatrix{ A \ar@<-2pt>[r]_f & B \ar@<-2pt>[l]_s }$ with kernel $k \colon X \to A$ in a pointed finitely complete category is \emph{strong} if $k$ and $s$ are jointly strongly epimorphic. It is \emph{stably strong} if every pullback of it along any morphism $g \colon C \to B$ is strong.
\end{definition}

\begin{definition}[\cite{MRVdL objects}]
An object $Y$ of a finitely complete category $\C$ is
\begin{itemize}
\item[(1)] a \emph{strongly unital object} if the point $\xymatrix{ Y \times Y \ar@<-2pt>[r]_-{\pi_2} & Y \ar@<-2pt>[l]_-{\langle 1, 1 \rangle} }$ is stably strong;

\item[(2)] a \emph{Mal'tsev object} if, for every pullback of split
epimorphisms over $Y$ as in the following diagram
\[
\vcenter{\xymatrix@!0@=5em{ A\times_{Y}C \ar@<-.5ex>[d]_{\pi_A}
\ar@<-.5ex>[r]_(.7){\pi_C} & C \ar@<-.5ex>[d]_g
\ar@<-.5ex>[l]_-{\langle sg,1_C \rangle} \\
A \ar@<-.5ex>[u]_(.4){\langle 1_A,tf \rangle} \ar@<-.5ex>[r]_f &
Y, \ar@<-.5ex>[l]_s \ar@<-.5ex>[u]_t }} \] the morphisms $\langle
1_{A}, tf \rangle$ and $\langle sg, 1_{C} \rangle$ are jointly
strongly epimorphic;

\item[(3)] a \emph{protomodular} object if every point over $Y$ is stably strong.
\end{itemize}
\end{definition}

In a unital category, an object is strongly unital if, and only if, it is \emph{gregarious} in the sense of \cite[Definition 1.9.1]{BBbook}. Strongly unital objects in the category $\Mon$ of monoids are then characterized in \cite[Proposition 1.9.2]{BBbook}: a monoid $M$ is strongly unital if and only if for any $m$ in $M$ there exist $u, v \in M$ such that $u + m + v = 0$. Clearly, every group satisfies this condition. An important fact for us is that every cancellative, strongly unital monoid is a group. The proof of this fact that we present here was suggested to us by Alfredo Costa, to whom we are grateful.

\begin{lemma} \label{gregarious + cancellative = group}
Every cancellative, strongly unital monoid $M$ is a group.
\end{lemma}

\begin{proof}
Let $m \in M$ and let $u,v \in M$ be such that $u+m+v=0$. We show that $v+u$ is the inverse of $m$. From $u+m+v=0$ we get
\[ u+m+v= 0 = u+m+v+u+m+v; \]
cancelling $u$ on the left and $v$ on the right we get
\[ m = 0 + m = m+v+u+m. \]
Now, cancelling $m$ on the right we obtain $m+v+u = 0$. Similarly, one gets $v+u+m = 0$.
\end{proof}

\begin{theorem}
For a right-preordered group $Y$, the following conditions are equivalent:
\begin{tfae}
\item $Y$ is a protomodular object in $\ROrdGrp$;
\item $Y$ is a Mal'tsev object in $\ROrdGrp$;
\item $Y$ is a strongly unital object in $\ROrdGrp$;
\item $P_Y$ is a group;
\item the preorder relation on $Y$ is an equivalence relation.
\end{tfae}
\end{theorem}

\begin{proof}
The equivalence between (iv) and (v) is obvious.

(iv) $\Rightarrow$ (i): since, in particular, every pair of morphisms in $\ROrdGrp$ with the same codomain is jointly strongly epimorphic in $\ROrdGrp$ provided that it is jointly strongly epimorphic in $\Grp$ and its restriction to the positive cones is jointly strongly epimorphic in $\Mon$, we can use the argument of \cite[Theorem 4.6]{CMFM}.

(i) $\Rightarrow$ (ii) follows from \cite[Proposition 7.2]{MRVdL objects}.

(ii) $\Rightarrow$ (iii) follows from \cite[Proposition 6.3]{MRVdL objects}, because \ROrdGrp\ is a regular category.

(iii) $\Rightarrow$ (iv): according to Lemma \ref{gregarious + cancellative = group}, we only need to show that $P_Y$ is a gregarious monoid. Suppose there is an element $b \in P_Y$ for which there are no $u, v \in M$ with $u + b + v = 0$. Let $X = \langle b \rangle$ be the subgroup of $Y$ generated by $b$, with the induced preorder, and $j \colon X \hookrightarrow Y$ the inclusion. As a (right-)preordered group, $X$ is isomorphic to $\mathbb{Z}$ with its usual order, namely $P_X = \{ nb \ | \ n \in \mathbb{N} \}$. Consider then the following right-hand side pullback in $\ROrdGrp$:
\[ \xymatrix{ Y \ar@{=}[d] \ar[r]^-{\langle 1, 0 \rangle} & Y \times X \ar[d]^{1 \times j} \ar@<-2pt>[r]_-{\pi_2} & X \ar@<-2pt>[l]_-{\langle j, 1 \rangle} \ar[d]^j \\
Y \ar[r]^-{\langle 1, 0 \rangle} & Y \times Y \ar@<-2pt>[r]_-{\pi_2} & Y. \ar@<-2pt>[l]_-{\langle 1, 1 \rangle} } \]
We show that the positive cone of $Y \times X$ (which is $P_Y \times P_X$) contains strictly the submonoid $P$ of $Y \times X$ generated by $\langle 1, 0 \rangle(P_Y)$ and $\langle j, 1 \rangle(P_X)$; this would prove that the upper point in the pullback above is not strong, contradicting the assumption. In particular, we show that $(0, b) \notin P$. The elements of $P$ are of the form
\[ (y_1, 0) + (n_1 b, n_1 b) + (y_2, 0) + (n_2 b, n_2 b) + \ldots +  (y_k, 0) + (n_k b, n_k b) \]
for some $k, n_i \in \mathbb{N}$, $y_i \in Y$. If $(0, b) \in P$, then we should have
\[ \begin{cases}
y_1 + n_1 b + y_2 + n_2 b + \ldots + y_k + n_k b = 0 \\
n_1 + n_2 + \ldots + n_k = 1
\end{cases} \]
but this is possible only if there is a unique $i \in \{1, \ldots k\}$ such that $n_i = 1$, while all the others $n_j$'s are $0$. So we would get
\[ y_1 + y_2 + \ldots + y_i + b + y_{i+1} + \ldots + y_k = 0, \]
which is against our assumption on $b$.
\end{proof}

\section{Split extensions in $\ROrdGrp$} \label{section split extensions}

Any split extension
\begin{equation}\label{eq:split}
\xymatrix{ X \ar[r]^k & A \ar@<-2pt>[r]_p & B \ar@<-2pt>[l]_s}
\end{equation}
 in $\ROrdGrp$ is in particular a split extension in $\Grp$. Hence $A$, as a group, is isomorphic to the semidirect product $X \rtimes_{\varphi} B$ w.r.t.
the action $\varphi$ of $B$ on $X$ given by $\varphi_b(x) =k^{-1}(s(b)+k(x)-s(b))$. Furthermore, in $\Grp$ the split extension \eqref{eq:split} is isomorphic to $\xymatrix{ X \ar[r]^-{\langle 1,0\rangle} & X\rtimes_\varphi B \ar@<-2pt>[r]_-{\pi_B} & B \ar@<-2pt>[l]_-{\langle 0,1\rangle}}$. Hence every split extension \eqref{eq:split} in $\ROrdGrp$ is isomorphic to a split extension of the form
\begin{equation}\label{eq:sdp}
\xymatrix{ X \ar[r]^-{\langle 1,0\rangle} & X\rtimes_\varphi B \ar@<-2pt>[r]_-{\pi_B} & B \ar@<-2pt>[l]_-{\langle 0,1\rangle}}
\end{equation}
where $X\rtimes_\varphi B$ is equipped with a preorder which makes \eqref{eq:sdp} a split extension in $\ROrdGrp$. We call such \emph{right-preorders} \emph{compatible}. We already know from \cite[Section 5]{CMFM} that, given a split extension \eqref{eq:sdp} in $\Grp$ with $X$ and $B$ preordered groups, there may be more than one compatible preorder in $X\rtimes_\varphi B$, and there may be none. It follows immediately that in $\ROrdGrp$ there is no uniqueness of compatible right-preorders. We will see next that, analogously, existence is not guaranteed.

As in $\OrdGrp$ (see \cite{CMFM, CR} for details), the positive cone $P$ of a compatible right-preorder must contain
\[P_\prod=P_X\times P_B=\{(x,b)\in X\rtimes_\varphi B\,|\,x\geq 0 \mbox{ and }b\geq 0\},\]
and must be contained in
\[P_\lex=\{(x,b)\in X\rtimes_\varphi B\,|\,b > 0\mbox{ or }(b\sim 0\mbox{ and }x\geq 0)\}.\]

\begin{prop}
If $P$ is a compatible cone in $X\rtimes_\varphi B$, then $P_\prod\subseteq P\subseteq P_\lex$.
\end{prop}
\begin{proof}
The morphism $\langle 1,0\rangle\colon X\to X\rtimes_\varphi B$ is a kernel in $\ROrdGrp$ if and only if $(x\in P_X\,\Leftrightarrow\,(x,0)\in P)$, while $\langle 0,1\rangle\colon B\to X\rtimes_\varphi B$ is monotone if and only if $(0,b)\in P$ whenever $b\in P_B$. Hence $P_\prod\subseteq P$. Moreover, monotonicity of $\pi_B$ gives that $b\in P_B$ whenever $(x,b)\in P$. If $b\sim 0$ then $-b\in P_B$, and so $(0,-b)\in P$; therefore, if $(x,b)\in P$, then $(x,0)=(x,b)+(0,-b)\in P$, hence $x\in P_X$.
\end{proof}

\begin{theorem}
For a split extension \eqref{eq:sdp} in $\Grp$ with $(X,P_X)$ and $(B,P_B)$ right-preordered groups, the following conditions are equivalent:
\begin{tfae}
\item $P_\lex$ is a compatible right-preorder;
\item there is a compatible right-preorder $P$;
\item $(\forall b\in P_B)\;(\exists b'\in P_B\;b+b'\sim 0)\Rightarrow \varphi_b \mbox{ monotone}$.
\end{tfae}
\end{theorem}
\begin{proof}
(i) $\Rightarrow$ (ii) is obvious.

(ii) $\Rightarrow$ (iii): suppose $b,b'\in P_B$, and let $x\in P_X$; then $(0,b)+(x,b')=(\varphi_b(x),b+b')$ must belong to $P$. If $b+b'\sim 0$ then necessarily $\varphi_b(x)\geq 0$ as claimed.

(iii) $\Rightarrow$ (i): we need to show that $P_\lex$ is a submonoid of $X\rtimes_\varphi B$; let $(x,b),(y,b')\in P_\lex$; then $(x,b)+(y,b')=(x+\varphi_b(y),b+b')$ belongs trivially to $P_\lex$ when $b+b'\not\sim 0$; if $b+b'\sim 0$ then, by (iii), $\varphi_b(y)\geq 0$ and consequently $x+\varphi_b(y)\geq\varphi_b(y)\geq 0$.
\end{proof}

\begin{corollary}\label{cor:ord}
If $B$ is a right-ordered group (i.e., with an antisymmetric preorder), then the right-preorder $P_\lex$ is compatible.
\end{corollary}

One can also identify the split extensions admitting $P_\prod$ as a compatible cone.
\begin{prop}
For a split extension \eqref{eq:sdp} in $\Grp$, with $(X,P_X)$ and $(B,P_B)$ right-preordered groups, the following conditions are equivalent:
\begin{tfae}
\item $P_\prod$ is a compatible right-preorder.
\item $(\forall b\in P_B)\;\varphi_b$ is monotone.
\end{tfae}
\end{prop}

\begin{proof}
(i) $\Rightarrow$ (ii): If $b\in P_B$ and $x\in P_X$, then $(0,b)+(x,0)=(\varphi_b(x),b)\in P_\prod$, hence $\varphi_b(x)\in P_X$ and therefore $\varphi_b$ is monotone.

(ii) $\Rightarrow$ (i): Given $x,y\in P_X$ and $a,b\in P_B$, $(x,a)+(y,b)=(x+\varphi_a(y),a+b)\in P_\prod$ since $a+b\in P_B$ and $x+\varphi_b(y)\in P_X$ by (ii).
\end{proof}

\begin{remark}
We do not know whether $(b\geq 0\;\wedge\;\exists b'\geq 0\;b+b'\sim 0)$ is equivalent to $b\sim 0$. This is the case when $B$ is a preordered group.
\end{remark}

A comparison of Corollary \ref{cor:ord} above with Theorem 3.2 of \cite{CR} gives us the following

\begin{corollary}
When $B$ is an ordered group and $\varphi_b$ is not monotone for some $b\in B$, then $P_\lex$ -- and every possible compatible right-preorder on $X\rtimes_\varphi B$ -- makes it a right-preordered group but not a preordered one.
\end{corollary}

Examples \ref{ex:rpo}(2,3) below are concrete instances of the situation described in the previous corollary.

\begin{examples}\label{ex:rpo}
\begin{enumerate}
\item As shown in \cite[Example 5.8]{CMFM}, the split extension
\[ \xymatrix{ (\mathbb{Z},\mathbb{N}) \ar[r]^-{\langle 1, 0 \rangle} & \mathbb{Z} \times \mathbb{Z} \ar@<-2pt>[r]_-{\pi_2} & (\mathbb{Z}, \mathbb{N}) \ar@<-2pt>[l]_-{\langle 0, 1 \rangle} }\]
has an uncountable number of compatible (right-)preorders (here every right-preorder is a preorder on $\bkZ\times\bkZ$ since the group is abelian).
\item Consider the split extension
\begin{equation}\label{eq:zz}\xymatrix{ (\mathbb{Z},P) \ar[r]^-{\langle 1, 0 \rangle} & \mathbb{Z} \rtimes_\varphi \mathbb{Z} \ar@<-2pt>[r]_-{\pi_2} & (\mathbb{Z}, \bkZ) \ar@<-2pt>[l]_-{\langle 0, 1 \rangle} }
\end{equation}
with $\varphi_b(x)=(-1)^b x$.

When $P=\bkN$ then $\varphi_b$ is not monotone when $b$ is an odd number (although every $b\sim 0$), hence there is no right-preorder making \eqref{eq:zz} a split extension in $\ROrdGrp$.

When $P=\{0\}$, $\varphi_b$ is trivially monotone, hence $P_\prod$, which coincides with $P_\lex$, is a compatible right-preorder. Note that there is no compatible preorder making $\bkZ\rtimes_\varphi \bkZ$ a preordered group, since the existence of such compatible preorder would imply that $\varphi_b\sim\id$ for every $b\in \bkZ$ (see \cite[Theorem 3.2]{CR}).

\item We can extend the previous situation, providing plenty of examples of right-(pre)ordered groups which are not (pre)ordered. Let us describe explicitly some of them. We recall that an action of a group $B$ on a group $X$ can be expressed as a group homomorphism from $B$ to the automorphism group $\mathrm{Aut}(X)$. When $X$ is the additive group of rationals, $\mathrm{Aut}(X) = (\mathbb{Q} \setminus \{0\}, \cdot)$. Hence, an action of $\mathbb{Z}$ on $\mathbb{Q}$, namely a group homomorphism $\mathbb{Z} \to \mathbb{Q} \setminus \{0\}$, is uniquely determined by the image of $1$, i.e. by a non-zero rational number $q$. Let us then consider the split extension
\[ \xymatrix{ (\mathbb{Q},\mathbb{Q}^+) \ar[r]^-{\langle 1, 0 \rangle} & \mathbb{Q} \rtimes_\varphi \mathbb{Z} \ar@<-2pt>[r]_-{\pi_2} & (\mathbb{Z}, \mathbb{N}) \ar@<-2pt>[l]_-{\langle 0, 1 \rangle} }\]
where $\varphi$ is the action determined by the non-zero rational number $q$, i.e. $\varphi_b(x) = q^b x$. The lexicographic preorder (actually, order) on $\mathbb{Q} \rtimes_\varphi \mathbb{Z}$ is always right-compatible, since $(\mathbb{Z}, \mathbb{N})$ is an ordered group (i.e. the preorder is antisymmetric), but when $q<0$ the map $\varphi_b$ is not monotone for all positive $b$, and so, according to \cite[Theorem 3.2]{CR}, there is no compatible preorder on $\mathbb{Q} \rtimes_\varphi \mathbb{Z}$.
\end{enumerate}
\end{examples}

\section{Further comments}

The passage from preordered groups to right-preordered groups as outlined here can be carried out to the more general context of $V$-groups, when $V$ is a commutative and unital quantale, as studied in \cite{CM}. We take this opportunity to point out that in the statement (ii) of Proposition 3.1 of our paper \cite{CM} a condition is missing. Indeed, a $V$-group is a group which has a $V$-category structure both left- and right-invariant under shifting, and this is the notion studied in \cite{CM}. Condition (ii) of Proposition 3.1 of \cite{CM} gives a notion of right-invariant $V$-group, or simply right-$V$-group, that may be interesting to explore in future work.

\section*{Acknowledgements}
The authors wish to express their gratitude to Vincenzo Marra for fruitful discussions on the topic of this paper, and to Alfredo Costa for suggesting us the proof of Lemma \ref{gregarious + cancellative = group}.

The first author acknowledges partial financial support by {\it Centro de Matemática da Universidade de Coimbra} (CMUC), funded by the Portuguese Government through FCT/MCTES, DOI 10.54499/UIDB/00324/2020.

The second author is member of the Gruppo Nazionale per le Strutture Algebriche, Geometriche e le loro Applicazioni (GNSAGA) dell'Istituto Nazionale di Alta Matematica ``Francesco Severi''.

\end{document}